\documentclass[a4paper,12pt,onecolumn]{article}
\usepackage{etex}

\usepackage[vmargin=2cm,hmargin=2cm,headheight=14.5pt,top=2cm,headsep=.5cm]{geometry}

\usepackage{bm}
\usepackage{empheq}
\usepackage{stackrel}
\usepackage{cases}
\usepackage{mathtools}
\usepackage{amsthm,amsmath,amscd}
\usepackage{makeidx}
\usepackage[all]{xy}
\usepackage{perpage}
\usepackage{tikz-cd}
\tikzcdset{every label/.append style = {font = \small}}
\usepackage[symbol]{footmisc}

\MakePerPage[2]{footnote}
\usepackage{graphicx}
\DeclareMathSizes{12}{12}{8}{6}

\usepackage{cite}
\usepackage{url}
\usepackage[charter]{mathdesign}
\usepackage{accents}
\usepackage{hyperref}

\newtheoremstyle{ptheorem}{1em}{0em}{\itshape}{}{\bfseries}{.}{.5em}{\thmname{#1}\thmnumber{ #2}\thmnote{ (\hspace{-.01pt}{#3})}}

\theoremstyle{ptheorem}

\newtheorem{thm}{Theorem}[section]
\newtheorem{pro}[thm]{Proposition}
\newtheorem{lem}[thm]{Lemma}
\newtheorem{cor}[thm]{Corollary}

\newtheoremstyle{hdef}{1em}{0em}{}{}{\bfseries}{.}{.5em}{\thmname{#1}\thmnumber{ #2}\thmnote{ (\hspace{-.01pt}{#3})}}
\theoremstyle{hdef}

\newtheorem{dfn}[thm]{Definition}
\newtheorem{rem}[thm]{Remark}

\makeatletter
\newtheoremstyle{premark}{1em}{0em}{
}{}{\scshape}{.}{.5em}{}
\makeatother

\theoremstyle{premark}

\newtheorem{exa}[thm]{Example}

\numberwithin{equation}{section}
\numberwithin{figure}{section}

\DeclareMathOperator{\sign}{sign}

\DeclareMathOperator{\Id}{Id}

\DeclareMathOperator{\dif}{d}

\newcommand{\cB}{{\mathcal B}}
\newcommand{\cC}{{\mathcal C}}

\newcommand{\cF}{{\mathcal F}}

\newcommand{\cO}{{\mathcal O}}
\newcommand{\cP}{{\mathcal P}}

\newcommand{\bN}{{\mathbb N}}

\newcommand{\bR}{{\mathbb R}}

\newcommand{\bZ}{{\mathbb Z}}

\renewcommand{\a}{\alpha}
\renewcommand{\b}{\beta}

\renewcommand{\l}{\lambda}
\newcommand{\e}{\epsilon}

\renewcommand{\phi}{\varphi}

\newcommand{\ol}{\overline}
\newcommand{\fa}{\forall}

\newcommand{\n}{{n\in\bN}}

\newcommand{\Ra}{\Rightarrow}

\newcommand{\nkp}{\enskip}

\newcommand{\sfa}{\nkp\fa}

\renewcommand{\(}{\left(}
\renewcommand{\)}{\right)}

\newcommand{\lil}{\lim\limits}

\newcommand{\olb}[1]{%
	\vbox{\offinterlineskip\ialign{\hfil##\hfil\cr $\rotatebox[origin=c]{90}{$]$}$\cr\noalign{\kern-.45ex}{$#1$}\cr}}}
\allowdisplaybreaks
\begin{document}
\title{Comparison results for first order linear operators with reflection and periodic boundary value conditions}

\author{
Alberto Cabada\thanks{Partially supported by FEDER and Ministerio de Educaci\'on y Ciencia, Spain, project MTM2010-15314} \, and F. Adri\'an F. Tojo\\
\normalsize
Departamento de An\'alise Ma\-te\-m\'a\-ti\-ca, Facultade de Matem\'aticas,\\ 
\normalsize Universidade de Santiago de Com\-pos\-te\-la, Spain.\\ 
\normalsize e-mail: alberto.cabada@usc.es, adrinusa@hotmail.com}
\date{}

\maketitle

\begin{abstract}
This work is devoted to the study of the first order operator $x'(t)+m\,x(-t)$ coupled with periodic boundary value conditions. We describe the eigenvalues of the operator and obtain the expression of its related Green's function in the non resonant case. We also obtain the range of the values of the real parameter $m$ for which the integral kernel, which provides the unique solution, has constant sign. In this way, we automatically establish maximum and anti-maximum principles for the equation. Some applications to the existence of nonlinear periodic boundary value problems are showed.
\end{abstract}

\noindent {\bf Keywords:}  Equations with reflection. Green's functions. Comparison principles. Periodic conditions.

\section{Introduction}
The study of functional differential equations with involutions can be traced back to the solution of the equation $x'(t)=x(\frac{1}{t})$ by Silberstein (see \cite{Sil}) in 1940. Wiener proves in \cite{Wie3} that the solutions of the Silberstein equation solve $t^2 x''(t)+x(t)=0$. On the other hand, by defining $y(t) = x(e^t)$, we conclude  that $x$ is a solution of the Silberstein equation if and only if $y'(t)=e^{-t} y(-t)$. This kind of equations are known as equations with reflection and, as \v{S}arkovski\u{\i} shows in \cite{Sar}, they have some applications to the stability of differential -- difference equations. Moreover this kind of equations has some interesting properties by itself, in fact it is not difficult to verify that the unique solution of the homogeneous harmonic oscillator $x''(t)+m^2\, x(t)=0$, coupled with the initial conditions $x(0)=x_0$, $x'(0)=-m\,x_0$, for any $x_0 \in \bR$, is the unique solution of the first order equation with reflection $x'(t)+m\, x(-t)=0$, $x(0)=x_0$ and vice-versa.

Wiener and Watkins study in \cite{Wie} the solution of the equation $x'(t)-a\, x(-t)=0$ with initial conditions. Equation $x'(t)+a\, x(t)+b\,x(-t)=g(t)$ has been treated by Piao in \cite{Pia, Pia2}. In \cite{Kul, Sha, Wie, Wat1, Wie2} some results are introduced to transform this kind of problems with involutions and initial conditions into second order ordinary differential equations with initial conditions or first order two dimensional systems, granting that the solution of the last will be a solution to the first. Furthermore, asymptotic properties and boundedness of the solutions of initial first order problems are studied in \cite{Wat2} and \cite{Aft} respectively. Second order boundary value problems have been considered in \cite{Gup, Gup2, Ore2, Wie2} for Dirichlet and Sturm-Liouville boundary value conditions, higher order equations has been studied in \cite{Ore}. Other techniques applied to problems with reflection of the argument can be found in \cite{And, Ma, Wie1}.

Despite all this progression of studies and to the best of our knowledge, the case of first order differential equations with reflection and periodic boundary value conditions has been disregarded so far. In this article and using some of the methods and results described in \cite{Cab1,Cab2}, we will find a solution to the equation $x'(t)+m\, x(-t)=h(t)$ with periodic boundary value conditions and then establish some properties of the solution. On the contrary to the majority of the previous mentioned papers, our approach consists on to study directly the first order functional equation and obtain the expression of the related Green's function.\par
In section 2 we present some ways in which differential problems with involutions can be reduced to ordinary differential problems. The expression of the Green's function mentioned before is obtained in section 3. Section 4 is devoted to the study of the range of the parameter $m \in \bR$ for which the sign of this function is constant. Finally, we present in section 5 some applications that allow us to ensure the existence of solutions for nonlinear problems. 

\section{Reduction of differential equations with involutions to ODE}
\begin{dfn} Let $A\subset\bR$, a function $\phi:A\to A$ such that $\phi\ne\Id$ and $\phi\circ\phi=\Id$ is called an \textbf{involution}.
\end{dfn}
Let us consider the problems
\begin{equation}\label{eqinv}
x'(t)=f(x(\phi(t))), \quad x(c)=x_c
\end{equation}
and
\begin{equation}\label{ode}
x''(t)=f'(f^{-1}(x'(t)))f(x(t))\phi'(t), \quad x(c)=x_c, \;x'(c)=f(x_c).
\end{equation}
Then we have the following Lemma:
\begin{lem}\label{lem1}
Let $(a,b)\subset\bR$ and let $f:(a,b)\to(a,b)$ be a diffeomorphism. Let $\phi\in\cC^1((a,b))$ be an involution. Let $c$ be a fixed point of $\phi$. Then $x$ is a solution of the first order differential equation with involution
(\ref{eqinv}) if and only if $x$ is a solution of the second order ordinary differential equation
(\ref{ode}).\end{lem}

We note that this Lemma improves Theorem 2.3 in \cite{Sha}. Also note that in the conditions of the Lemma, $\phi$ is decreasing and $c$ is the unique fixed point of which the existence is guaranteed (see the results in \cite{Wie}, which can be generalized to the interval $(a,b)$).

\begin{proof}
That those solutions of (\ref{eqinv}) are solutions of (\ref{eqinv}) is almost trivial. The boundary conditions are justified by the fact that $\phi(c)=c$. Since differentiating (\ref{eqinv}) we get
$$x''(t)=f'(x(\phi(t)))\,x'(\phi(t))\,\phi'(t)$$
and taking into account that $x'(\phi(t))=f(x(t))$ by (\ref{eqinv}), we obtain (\ref{ode}).\par
Conversely, let $x$ be a solution of (\ref{ode}). The equation implies that
\begin{equation}\label{e-f-1}
(f^{-1})'(x'(t))x''(t)=f(x(t))\phi'(t).\end{equation}
Integrating from $c$ to $t$ in (\ref{e-f-1}),
\begin{equation}\label{e-f-2}
f^{-1}(x'(t))-x_c=f^{-1}(x'(t))-f^{-1}(x'(c))=\int_c^tf(x(s))\phi'(s)\dif s\end{equation}
and thus, defining $g(s):=f(x(\phi(s)))-x'(s)$, we conclude from (\ref{e-f-2}) that
\begin{align*}
x'(t)&=f\(x_c+\int_c^tf(x(s))\phi'(s)\dif s\)=f\(x(\phi(t))+\int_c^t(f(x(s))-x'(\phi(s)))\phi'(s)\dif s\)\\
&=f\(x(\phi(t))+\int_c^{\phi(t)}(f(x(\phi(s)))-x'(s))\dif s\)=f\(x(\phi(t))+\int_c^{\phi(t)}g(s)\dif s\).
\end{align*}\par
Let us fix $t>c$ where $x(t)$ is defined. We will prove that (\ref{eqinv}) is satisfied in $[c,t]$ (the proof is done analogously for $t<c$). Recall that $\phi$ has to be decreasing, so $\phi(t)<c$. Also, since $f$ is a diffeomorphism, the derivative of $f$ is bounded on $[c,t]$, so $f$ is Lipschitz on $[c,t]$. Since $f$, $x$, $x'$ and $\phi'$ are continuous, we can define
$$K_1:=\inf\left\{\a\in\bR^+\ :\ \left|f\(x(\phi(r))+\int_c^{\phi(r)}g(s)\dif s\)-f(x(\phi(r)))\right|\le\a\left|\int_c^{\phi(r)}g(s)\dif s\right|\sfa r\in[c,t]\right\}$$
and
$$K_2:=\inf\left\{\a\in\bR^+\ :\ \left|f\(x(r)+\int_c^{r}g(s)\dif s\)-f(x(r))\right|\le\a\left|\int_c^{r}g(s)\dif s\right|\sfa r\in[c,t]\right\}$$
Let $K=\max\{K_1,K_2\}$. Now,
\begin{align*}
|g(t)|&=\left|f\(x(\phi(t))+\int_c^{\phi(t)}g(s)\dif s\)-f(x(\phi(t)))\right|\le K\left|\int_c^{\phi(t)}g(s)\dif s\right|\\&\le-K\int_c^{\phi(t)}|g(s)|\dif s=-K\int_c^{t}|g(\phi(s))|\phi'(s)\dif s.
\end{align*}
Applying this inequality at $r=\phi(s)$ inside the integral we deduce that
\begin{align*}|g(t)|&\le-K\int_c^{t}K\left|\int_c^{s}g(r)\dif r\right|\phi'(s)\dif s\le-K^2\int_c^{t}\int_c^{t}|g(r)|\dif r\nkp\phi'(s)\dif s\\&=K^2|\phi(t)-\phi(c)|\int_c^t|g(r)|\dif r\le K^2(c-a)\int_c^t|g(r)|\dif r.
\end{align*}
Thus, by Gr\"onwall's Lemma, $g(t)=0$ and hence (\ref{eqinv}) is satisfied for all $t<b$ where $x$ is defined.
\end{proof}
Notice that, as an immediate consequence of this result, we have that the unique solution of the equation 
$$x''(t)=-\sqrt{1+(x'(t))^2}\, \sinh{x(t)}, \quad x(0)=x_0,\; x'(0)=\sinh{x_0},$$
coincide with the unique solution of
$$x'(t)=\sinh{x(-t)}, \quad x(0)=x_0.$$

Furthermore, Lemma \ref{lem1} can be extended, with a very similar proof, to the case with periodic boundary value conditions.
Let us consider the equations
\begin{equation}\label{eqinvb}
x'(t)=f(x(\phi(t))), \quad x(a)=x(b)
\end{equation}
and
\begin{equation}\label{odeb}
x''(t)=f'(f^{-1}(x'(t)))f(x(t))\phi'(t), \quad x(a)=x(b)=f^{-1}(x'(a)).
\end{equation}
\begin{lem}\label{lem2}
Let $[a,b]\subset\bR$ and let $f:[a,b]\to[a,b]$ be a diffeomorphism. Let $\phi\in\cC^1([a,b])$ be an involution such that $\phi([a,b])=[a,b]$. Then $x$ is a solution of the first order differential equation with involution
(\ref{eqinvb}) if and only if $x$ is a solution of the second order ordinary differential equation
(\ref{odeb}).\end{lem}
\begin{proof}
Let $x$ be a solution of (\ref{eqinvb}). Since $\phi(a)=b$ we trivially get that $x$ is a solution of (\ref{odeb}).\par
Let $x$ be a solution of (\ref{odeb}). As in the proof of the previous lemma, we have that
$$x'(t)=f\(x(\phi(t))+\int_b^{\phi(t)}g(s)\dif s\),$$
where $g(s):=f(x(\phi(s)))-x'(s)$.\par
Let $K_1$, $K_2$ be as in the proof of Lemma $\ref{lem1}$ but changing $c$ by $a$ and $[c,t]$ by $[a,b]$. Let $K_1'$, $K_2'$ be as $K_1$, $K_2$ but changing $c$ by $b$. Let $K=\max\{K_1,K_2,K_1',K_2'\}$. Then, for $t$ in $[a,b]$,
\begin{align*}
|g(t)|&\le K\left|\int_b^{\phi(t)}g(s)\dif s\right|\le-K\int_a^{t}|g(\phi(s))|\phi'(s)\dif s\le-K\int_a^{t}K\left|\int_a^{s}g(r)\dif r\right|\phi'(s)\dif s\\&\le K^2|\phi(t)-\phi(a)|\int_a^t|g(r)|\dif r\le K^2(b-a)\int_a^t|g(r)|\dif r,
\end{align*}
and we conclude analogously to the other proof.
\end{proof}
\begin{rem} Condition $x(a)=x(b)=f^{-1}(x'(a))$ in Lemma \ref{lem2} can be replaced by $x(a)=x(b)=f^{-1}(x'(b))$. The proof in this case is analogous.
\end{rem}
Let $I:=[-T,T]\subset\bR$. In the case of a problem of the kind
\begin{equation}\label{eqinv2}
x'(t)=f(t,x(-t),x(t)), \quad x(-T)=x(T)
\end{equation}
we can think in other methods to reduce the problem to a system of first order ODE which will not be equivalent to (\ref{ode}). It is a known fact that any real function $g:I\to\bR$ can be expressed uniquely as $g=g_e+g_o$ where $g_e,g_o:I\to\bR$ are an even and an odd function respectively. $g_e$ and $g_o$ can be expressed as
$$g_e(t)=\frac{g(t)+g(-t)}{2},\quad g_o(t)=\frac{g(t)-g(-t)}{2}\quad \mbox{for all } \; t\in I.$$
In other words, if $\mathcal{D}(I)$, $\mathcal{E}(I)$, $\mathcal{O}(I)$ are, respectively, the real vector space of differentiable functions on $I$, the vector space of even differentiable functions on I and the vector space of odd differentiable functions on $I$, then $\mathcal{D}(I)=\mathcal{E}(I)\oplus\cO(I)$. Furthermore, the differential operator acts linearly on such space as
\begin{equation}\label{decomposition}\begin{CD}\mathcal{E}(I)\oplus\cO(I) @>D>> \mathcal{E}(I)\oplus\cO(I)\\
(\,g\,,\,h\,) @>>> \(\begin{smallmatrix}0 & D \\ D & 0\end{smallmatrix}\)\(\begin{smallmatrix}\ g \\ h\end{smallmatrix}\)=(h',g')\end{CD}\quad,\end{equation}
which allows us to stablish a system with two differential equations.\par
If we consider now the endomorphism $\xi:\bR^3\to\bR^3$ defined as
$$\xi(t,z,w)=\(t,z-w,z+w\)\sfa z,w\in\bR,$$
with inverse
$$\xi^{-1}(t,y,x)=\(t,\frac{x+y}{2},\frac{x-y}{2}\)\sfa \;x,\,y\in\bR.$$

It is clear that
$$f(t,x(-t),x(t))=(f\circ\xi)(t,x_e(t),x_o(t))\text{ and } f(-t,x(t),x(-t))=(f\circ\xi)(-t,x_e(t),-x_o(t)).$$

On the other hand, we define
\begin{align*}
g_e(t):=f_e(t,x(-t),x(t)) & =\frac{f(t,x(-t),x(t))+f(-t,x(t),x(-t))}{2}\\ & =\frac{(f\circ\xi)(t,x_e(t),x_o(t))+(f\circ\xi)(-t,x_e(t),-x_o(t))}{2}\end{align*}
and
$$g_o(t):=f_o(t,x(-t),x(t)) =\frac{(f\circ\xi)(t,x_e(t),x_o(t))-(f\circ\xi)(-t,x_e(t),-x_o(t))}{2},$$
which are an even and an odd function respectively. Furthermore, since $x_e$ is even, $x_e(-T)=x_e(T)$ and since $x_o$ is odd, $x_o(-T)=-x_o(T)$. Taking into account (\ref{decomposition}), we can state the following theorem.
\begin{thm}\label{thmeo} If $x$ is a solution for problem (\ref{eqinv2}) and $y(t)=x(-t)$, then $(z,w):I\to\bR^2$ satisfying $(t,z,w)=\xi^{-1}(t,y,x)$ is a solution for the system of ODE with boundary conditions
\begin{equation}\label{2eq}
\begin{aligned}
z'(t) & = \frac{(f\circ\xi)(t,z(t),w(t))-(f\circ\xi)(-t,z(t),-w(t))}{2},\quad t\in I,\\
w'(t) & = \frac{(f\circ\xi)(t,z(t),w(t))+(f\circ\xi)(-t,z(t),-w(t))}{2},\quad t\in I,\\
(z,w)(-T) & =(z,-w)(T).
\end{aligned}
\end{equation}
\end{thm}
We can take this one step further trying to ``undo'' what we did:
\begin{align*}
x'(t) & =(z+w)'(t)=(f\circ\xi)(t,z(t),w(t))=f(t,y(t),x(t)),\\
y'(t) &=(z-w)'(t)=-(f\circ\xi)(-t,z(t),-w(t))=-f(-t,x(t),y(t)),\\
(y,x)(-T)& =((z-w)(-T),(z+w)(-T))=((z+w)(T),(z-w)(T))=(x,y)(T).
\end{align*}
We get then the following result.
\begin{pro}\label{proxy} $(z,w)$ is a solution for problem (\ref{2eq}) if and only if $(y,x)$ such that $\xi(t,z,w)=(t,y,x)$ is a solution for the system of ODE with boundary conditions
\begin{equation}\label{2eq2}
\begin{aligned}
x'(t) & =f(t,y(t),x(t)),\\
y'(t) &=-f(-t,x(t),y(t)),\\
(y,x)(-T)& =(x,y)(T).
\end{aligned}
\end{equation}
\end{pro}
The next corollary can also be obtained in a straightforward way without going trough problem (\ref{2eq}).
\begin{cor}\label{corxy}If $x$ is a solution for problem (\ref{eqinv2}) and $y(t)=x(-t)$, then $(y,x):I\to\bR^2$ is a solution for the problem (\ref{2eq2}).
\end{cor}
Solving problems (\ref{2eq}) or (\ref{2eq2}) we can check whether $x$, obtained from the relation $(t,y,x)=\xi(t,z,w)$ is a solution to problem (\ref{eqinv2}). Unfortunately, not every solution of (\ref{2eq}) --or (\ref{2eq2})-- is a solution of (\ref{eqinv2}), as we show in the following example.
\begin{exa} Consider the problem
\begin{equation}
\label{e-ex}
x'(t)=x(t)\,x(-t), \; tÊ\in I; \quad x(-T)=x(T).
\end{equation}

Using Proposition \ref{proxy}, we know that the solutions of (\ref{e-ex}) are those of problem 
\begin{equation}\label{eqexa}\begin{aligned}
x'(t) & =x(t)\,y(t),\quad t \in I;\\
y'(t) & =-x(t)\,y(t),\quad t \in I;\\
(y,x)(-T)& =(x,y)(T).
\end{aligned}\end{equation}
It is easy to check that the only solutions of problem (\ref{eqexa}) defined on $I$ are of the kind
$$(x,y)=\(\frac{ce^{ct}}{e^{ct}+1},\frac{c}{e^{ct}+1}\),$$
with $c \in \bR$. However, if $c \neq 0$, $x(T)\ne x(-T)$ and, as consequence, no one of them is a solution of (\ref{e-ex}). 

Moreover, using Proposition \ref{proxy} again, we conclude that $x\equiv0$ is the only solution of  (\ref{e-ex}).
\end{exa}
\par
In a completely analogous way, we can study the problem
\begin{equation}\label{eqinv3}
x'(t)=f(t,x(-t),x(t)), \quad x(0)=x_0.
\end{equation}
In such a case we would have the following versions of the previous results:
\begin{thm}\label{thmeo2} If $x:(-\e,\e)\to\bR$ is a solution of problem (\ref{eqinv3}) and $y(t)=x(-t)$, then $(z,w):(-\e,\e)\to\bR^2$ satisfying $(t,z,w)=\xi^{-1}(t,y,x)$ is a solution of the system of ODE with initial conditions
\begin{equation}\label{2eq3}
\begin{aligned}
z'(t) & = \frac{(f\circ\xi)(t,z(t),w(t))-(f\circ\xi)(-t,z(t),-w(t))}{2},\quad t\in I,\\
w'(t) & = \frac{(f\circ\xi)(t,z(t),w(t))+(f\circ\xi)(-t,z(t),-w(t))}{2},\quad t\in I,\\
(z,w)(0) & =(x_0,0).
\end{aligned}
\end{equation}
\end{thm}
\begin{pro}\label{proxy2} $(z,w)$ is a solution of problem (\ref{2eq3}) if and only if $(y,x)$ such that $\xi(t,z,w)=(t,y,x)$ is a solution of the system of ODE with initial conditions
\begin{equation}\label{2eq4}
\begin{aligned}
x'(t) & =f(t,y(t),x(t)),\\
y'(t) &=-f(-t,x(t),y(t)),\\
(y,x)(0)& =(x_0,x_0).
\end{aligned}
\end{equation}
\end{pro}
\begin{cor}\label{corxy2} If $x:(-\e,\e)\to\bR$ is a solution of problem (\ref{eqinv3}) and $y(t)=x(-t)$, then $(y,x):(-\e,\e)\to\bR^2$ is a solution of problem (\ref{2eq4}).
\end{cor}
\begin{rem} The relation $y(t)=x(-t)$ is used in \cite{Wie2} to study conditions under which the problem
$$x'(t)=f(t,x(t),x(-t)),\quad t\in\bR$$
has a unique bounded solution.
\end{rem}

\section{Solution of the equation $x'(t)+m\, x(-t)=h(t)$}

In this section we will solve a first order linear equation with reflection coupled with periodic boundary value conditions. More concisely, we consider the following differential functional equation:
\begin{subequations}\label{eq1}
\begin{align}
\label{eq1a} x'(t)+m\, x(-t)&=h(t),\nkp t\in I,\\ \label{cn1} x(T)-x(-T)&=0,
\end{align}
\end{subequations}
where $m$ is a real non-zero constant, $T\in\bR^+$ and $h\in L^1(I)$.\par
In the homogeneous case, this is, $h\equiv0$, differentiating (\ref{eq1a}) and with some substitutions we arrive to the conclusion that any solution of (\ref{eq1}) has to satisfy the problem
\begin{equation}x''+m^2x=0, \quad x(T)-x(-T)= x'(T)-x'(-T)=0.
\end{equation}
Consider now the following ordinary differential equation with homogeneous boundary conditions
\begin{subequations}\label{eq2}
\begin{align}
\label{eq2a} x''(t)+m^2x(t)&=f(t),\nkp t\in I,\\ \label{cn2} x(T)-x(-T)&=0,\\x'(T)-x'(-T)&=0,
\end{align}
\end{subequations}
where $f$ is a continuous function on $I$.

There is much literature on how to solve this problem and the properties of the solution (see for instance \cite{Cab1,Cab2,Aga}). It is very well known that for all $m^2 \neq (k \pi/T)^2$, $k=0,1, \ldots$, problem (\ref{eq2}) has a unique solution given by the expression
$$u(t)=\int_{-T}^TG(t,s)f(s)\dif s,$$ 
where $G$ is the so-called Green's function. 

This function is unique insofar as it satisfies the following properties:
\begin{enumerate}
\item $G\in\cC(I^2,\bR)$,
\item $\frac{\partial G}{\partial t}$ and $\frac{\partial^2 G}{\partial t^2}$ exist and are continuous in $\{(t,s)\in I^2\ |\ s\ne t\}$,
\item $\frac{\partial G}{\partial t}(t,t^-)$ and $\frac{\partial G}{\partial t}(t,t^+)$ exist for all $t\in I$ and satisfy
$$\frac{\partial G}{\partial t}(t,t^-)-\frac{\partial G}{\partial t}(t,t^+)=1\sfa t\in I,$$
\item $\frac{\partial^2 G}{\partial t^2}+m^2G=0\text{ in }\{(t,s)\in I^2\ |\ s\ne t\},$
\item \begin{enumerate}
\item $G(T,s)=G(-T,s)\sfa s\in I$,
\item $\frac{\partial G}{\partial t}(T,s)=\frac{\partial G}{\partial t}(-T,s)\sfa s\in(-T,T)$.
\end{enumerate}
\end{enumerate}
The solution to (\ref{eq2}) is unique whenever $T\in\bR^+\backslash\{\frac{k\pi}{|m|}\}_{k\in\bN}$, so the solution to (\ref{eq1}) is unique in such a case. We will assume uniqueness conditions from now on.\par
The following proposition gives us some more properties of the Green's function for (\ref{eq2}).
\begin{pro} For all $t,s\in I$, the Green's function associated to problem (\ref{eq2}) satisfies the following properties as well:
\begin{enumerate}
\setcounter{enumi}{5}
\item $G(t,s)=G(s,t)$,
\item $G(t,s)=G(-t,-s)$,
\item $\frac{\partial G}{\partial t}(t,s)=\frac{\partial G}{\partial s}(s,t)$,
\item $\frac{\partial G}{\partial t}(t,s)=-\frac{\partial G}{\partial t}(-t,-s)$,
\item $\frac{\partial G}{\partial t}(t,s)=-\frac{\partial G}{\partial s}(t,s)$.
\end{enumerate}
\end{pro}
\begin{proof}
$(VI)$. The differential operator $L=\frac{\dif^2}{\dif t^2}+m^2$ associated to equation (\ref{eq2}) is self-adjoint, so in an analogous way to \cite[Chapter 33]{Aga},  we deduce that function $G$ is symmetric.

$(VII)$. Let $u$ be a solution to (\ref{eq2}) and define $v(t):=u(-t)$, then $v$ is a solution of problem (\ref{eq2}) with $f(-t)$ instead of $f(t)$. This way
$$v(t)=\int_{-T}^{T}G(t,s)f(-s)\dif s=\int_{-T}^{T}G(t,-s)f(s)\dif s,$$
but we have also
$$v(t)=u(-t)=\int_{-T}^{T}G(-t,s)f(s)\dif s,$$
therefore 
$$\int_{-T}^{T}[G(t,-s)-G(-t,s)]f(s)=0$$
 and, since continuous functions are dense in $L^2(I)$, $G(t,-s)=G(-t,s)$ on $I^2$, this is, 
 $$G(t,s)=G(-t,-s)\sfa t,s\in I.$$
 
To prove $(VIII)$ and $(IX)$ it is enough to differentiate $(VI)$ and $(VII)$ with respect to $t$.\par
$(X)$ Assume $f$ is differentiable. Let $u$ be a solution to (\ref{eq2}), then $u \in C^1(I)$ and $v \equiv u'$ is a solution of \begin{subequations}
\begin{align}
 x''(t)+m^2x(t)&=f'(t),\nkp t\in I,\\ x(T)-x(-T)&=0,\\x'(T)-x'(-T)&=f(T)-f(-T).
\end{align}
\end{subequations}\par
Therefore,
$$v(t)=\int_{-T}^TG(t,s)f'(s)\dif s-G(t,-T)[f(T)-f(-T)],$$
where the second term in the right hand side stands for the non-homegeneity of the boundary conditions and properties $(III)$, $(IV)$ and $(V)$ $(a)$.

Hence, from $(V) (a)$ and $(VI)$, we have that
\begin{eqnarray*}
v(t)&=&G(t,s)f(s)\big\rvert_{s=-T}^{s=T}-\int_{-T}^t\frac{\partial G}{\partial s}(t,s)f(s)\dif s-\int_{t}^T\frac{\partial G}{\partial s}(t,s)f(s)\dif s\\
&&-G(t,-T)[f(T)-f(-T)]\\
&=&G(t,-T)[f(T)-f(-T)]-\int_{-T}^T\frac{\partial G}{\partial s}(t,s)f(s)\dif s -G(t,-T)[f(T)-f(-T)].
\end{eqnarray*}

On the other hand,
$$v(t)=u'(t)=\frac{\dif}{\dif t}\int_{-T}^tG(t,s)f(s)\dif s+\frac{\dif}{\dif t}\int_{t}^TG(t,s)f(s)\dif s=\int_{-T}^T\frac{\partial G}{\partial t}(t,s)f(s)\dif s.$$

Since differentiable functions are dense in $L^2(I)$, we conclude that
$$\frac{\partial G}{\partial t}(t,s)=-\frac{\partial G}{\partial s}(t,s).$$
\end{proof}

Now we are in a position to prove the main result of this section, in which we deduce the expression of the Green's function related to problem (\ref{eq1}).
\begin{pro}\label{Greenf} Suppose that $m \neq k \, \pi/T$, $k \in \bZ$. Then problem (\ref{eq1}) has a unique solution given by the expression
\begin{equation}
\label{e-u}
u(t):=\int_{-T}^T\ol G(t,s)h(s)\dif s,
\end{equation}
where $$\ol{G}(t,s):=m\,G(t,-s)-\frac{\partial G}{\partial s}(t,s)$$
is called the \textbf{Green's function} related to problem (\ref{eq1}).
\end{pro}
\begin{proof}
As we have previously remarked, problem (\ref{eq1}) has at most one solution for all $m \neq k \, \pi/T$, $k \in \bZ$. The existence of solution follows from the Theorem of Alternative. Let's see that function $u$ defined in (\ref{e-u}) fulfills (\ref{eq1}) (we assume $t>0$, the other case is analogous):

\begin{eqnarray*}
u'(t)+m\, u(-t)&=&\frac{\dif}{\dif t}\int_{-T}^{-t}\ol G(t,s)h(s)\dif s+\frac{\dif}{\dif t}\int_{-t}^t\ol G(t,s)h(s)\dif s\\
&&+\frac{\dif}{\dif t}\int_{t}^T\ol G(t,s)h(s)\dif s+m\int_{-T}^T\ol G(-t,s)h(s)\dif s\\
&=&(\ol G(t,t^-)-\ol G(t,t^+))h(t)+\int_{-T}^T\left[m\frac{\partial G}{\partial t}(t,-s)-\frac{\partial^2 G}{\partial t\partial s}(t,s)\right]h(s)\dif s\\
&&+m\int_{-T}^T\left[mG(-t,-s)-\frac{\partial G}{\partial s}(-t,s)\right]h(s)\dif s.
\end{eqnarray*}

Using $(III)$ and $(X)$, we deduce that this last expression is equal to

$$h(t)+\int_{-T}^T\left[m\frac{\partial G}{\partial t}(t,-s)-\frac{\partial^2 G}{\partial t\partial s}(t,s)+m^2G(-t,-s)-m\frac{\partial G}{\partial s}(-t,s)\right]h(s)\dif s.$$
which is, by $(IV)$, $(VII)$, $(IX)$ and $(X)$, equal to
$$h(t)+\int_{-T}^T\left(m\left[\frac{\partial G}{\partial t}(t,-s)-\frac{\partial G}{\partial s}(-t,s)\right]
+\frac{\partial^2G}{\partial t^2}(t,s)+m^2G(t,s)
\right)h(s)\dif s=h(t).$$

Therefore, (\ref{eq1a}) is satisfied. 

Conditions $(V)$ and $(X)$ allow us to verify the contour condition:
$$u(T)-u(-T)=\int_{-T}^T\left[mG(T,-s)-\frac{\partial G}{\partial s}(T,s)-mG(-T,-s)+\frac{\partial G}{\partial s}(-T,s)\right]h(s)=0.$$
\end{proof}
\par
As the original Green function, $\ol G$ satisfies several properties.
\renewcommand{\labelenumi}{(\Roman{enumi}')}
\begin{pro} $\ol G$ satisfies the following properties:
\begin{enumerate}
\item $\frac{\partial \ol G}{\partial t}$ exists and is continuous in $\{(t,s)\in I^2\ |\ s\ne t\}$,
\item $\ol G(t,t^-)$ and $\ol G(t,t^+)$ exist for all $t\in I$ and satisfy $\ol G(t,t^-)-\ol G(t,t^+)=1\sfa t\in I$,
\item $\frac{\partial\ol G}{\partial t}(t,s)+m\ol G(-t,s)=0\text{ for a.e. }t,s\in I,\ s\ne t$,
\item $\ol G(T,s)=\ol G(-T,s)\sfa s\in (-T,T)$,
\item $\ol G(t,s)=\ol G(-s,-t)\sfa t,s\in I$.
\end{enumerate}
\end{pro}
\begin{proof} Properties $(I')$, $(II')$ and $(IV')$ are straightforward from the analogous properties for function $G$.\par
$(III')$. In the proof of Proposition \ref{Greenf} we implicitely showed that function $u$ defined in (\ref{e-u}), and thus the unique solution of (\ref{eq1}), satisfies
$$u'(t)=h(t)+\int_{-T}^T\frac{\partial\ol G}{\partial t}(t,s)h(s)\dif s.$$
Hence, since $u'(t)-h(t)+m\, u(-t)=0$,
$$\int_{-T}^T\frac{\partial\ol G}{\partial t}(t,s)h(s)\dif s+m\int_{-T}^T\ol G(t,s)h(s)\dif s=0$$
this is,
$$\int_{-T}^T\left[\frac{\partial\ol G}{\partial t}(t,s)+m\ol G(t,s)\right]h(s)\dif s=0\text{ for all } h\in L^1(I),$$
and thus $$\frac{\partial\ol G}{\partial t}(t,s)+m\ol G(-t,s)=0\text{ for a.e. }t,s\in I,\ s\ne t.$$\par

$(IV')$.  This result is proven using properties $(VI)-(X)$:
\begin{align*}\ol G(-s,-t) & =mG(-s,t)-\frac{\partial G}{\partial s}(-s,-t)=mG(t,-s)+\frac{\partial G}{\partial t}(-s,-t) \\ & =mG(t,-s)-\frac{\partial G}{\partial t}(s,t)=mG(t,-s)-\frac{\partial G}{\partial s}(t,s)=\ol G(t,s).
\end{align*}
\end{proof}
\renewcommand{\labelenumi}{(\arabic{enumi})}
\begin{rem} Due to the expression of $G$ given in next section, properties $(II)$ and $(I')$ can be improved by adding that $G$ and $\ol G$ are analytic on $\{(t,s)\in I^2\ |\ s\ne t\}$ and $\{(t,s)\in I^2\ |\ |s|\ne |t|\}$ respectively.
\end{rem} 
\par Using properties $(II')-(V')$ we obtain the following corollary of Proposition \ref{Greenf}.
\begin{cor}
Suppose that $m \neq k \, \pi/T$, $k \in \bZ$. Then the problem 
\begin{subequations}
\begin{align}
 x'(t)+m\, x(-t)&=h(t),\nkp t\in I:=[-T,T],\\  x(-T)-x(T)&=\l,
\end{align}
\end{subequations}
with $\l\in\bR$ has a unique solution given by the expression
\begin{equation}
u(t):=\int_{-T}^T\ol G(t,s)h(s)\dif s+\l\ol G(t,-T).
\end{equation}
\end{cor}

\section{Constant sign of function $\ol G$}

We will now give a result on the positivity or negativity of the Green's function for problem (\ref{eq1}). In order to achieve this, we need a new lemma and the explicit expression of the function $\ol G$.\par
Let $\a:=mT$ and $\ol G_\a$ be the Green's function for problem (\ref{eq1}) for a particular value of the parameter $\a$. Note that $\sign(\a)=\sign(m)$ because $T$ is always positive.

\begin{lem}\label{Gop} $\ol G_\a(t,s)=-\ol G_{-\a}(-t,-s)\sfa t,s\in I$.
\end{lem}
\begin{proof}
Let $u(t):=\int_{-T}^T\ol G_\a(t,s)h(s)\dif s$ be a solution to (\ref{eq1}). Let $v(t):=-u(-t)$. Then $v'(t)-m\, v(-t)=u'(-t)+m\, u(t)=h(-t)$, and therefore $v(t)=\int_{-T}^T\ol G_{-\a}(t,s)h(-s)\dif s$. On the other hand, by definition of $v$,
$$v(t)=-\int_{-T}^T\ol G_\a(-t,s)h(s)\dif s=-\int_{-T}^T\ol G_\a(-t,-s)h(-s)\dif s,$$
therefore we can conclude that $\ol G_\a(t,s)=-\ol G_{-\a}(-t,-s)$ for all $t,\;s\in I$.
\end{proof}
\begin{cor} $\ol G_\a$ is positive if and only if $\ol G_{-\a}$ is negative on $I^2$.
\end{cor}
With this corollary, to make a complete study of the positivity and negativity of the function, it is enough to find out for what values $\a\in\bR^+$ function $\ol G$ is positive and for which is not. This will be very useful to state maximum and anti-maximum principles for (\ref{eq1}) due to the way we express its solution as an integral operator with kernel $\ol G$.\par
Using the algorithm described in \cite{Cab2} we can obtain the explicit expression of $G$:
\begin{equation}
2m\sin(mT)G(t,s)=\begin{cases} \cos m(T+s-t) & \text{if}\quad s\le t,\\\cos m(T-s+t) & \text{if}\quad s>t.\end{cases}
\end{equation}
Therefore,
\begin{equation}\label{gbarra}
2\sin(mT)\ol G(t,s)=\begin{cases} \cos m(T-s-t)+\sin m(T+s-t) & \text{if}\quad -t\le s<t,\\\cos m(T-s-t)-\sin m(T-s+t) & \text{if}\quad -s\le t<s,\\\cos m(T+s+t)+\sin m(T+s-t) & \text{if}\quad -|t|>s,\\\cos m(T+s+t)-\sin m(T-s+t) & \text{if}\quad t<-|s|.\end{cases}
\end{equation}
Realize that $\ol G$ is continuous in $\{(t,s)\in I^2\ |\ t\ne s\}$. Making the change of variables $t=Tz$, $s=Ty$, we can simplify this expression to
\begin{equation}
2\sin(\a)\ol G(z,y)=\begin{cases} \cos \a(1-y-z)+\sin \a(1+y-z) & \text{if}\quad -z\le y <z,\\\cos \a(1-y-z)-\sin \a(1-y+z) & \text{if}\quad -y\le z <y,\\\cos \a(1+y+z)+\sin \a(1+y-z) & \text{if}\quad -|z|>y,\\\cos \a(1+y+z)-\sin \a(1-y+z) & \text{if}\quad z<-|y|.\end{cases}
\end{equation}
Using the trigonometric identity $\cos(a-b)\pm\sin(a+b)=(\cos a\pm\sin a)(\cos b\pm\sin b)$, we can factorise this expression as follows:
\begin{equation}\label{eqgb}
2\sin(\a)\ol G(z,y)=
\begin{cases}
[\cos\a(1-z)+\sin\a(1-z)][\sin\a y+\cos\a y]& \text{if}\quad -z\le y <z, \\
[\cos\a z-\sin\a z][\sin\a(y-1)+\cos\a(y-1)] & \text{if}\quad -y\le z <y, \\
[\cos\a(1+y) +\sin\a(1+y)][\cos\a z-\sin\a z]& \text{if}\quad -|z|>y, \\
[\cos\a y+\sin\a y][\cos\a(z+1)-\sin\a(z+1)]& \text{if}\quad z<-|y|.
\end{cases}
\end{equation}
Note that
\begin{subequations}\label{trif}
\begin{align}
\cos\xi+\sin\xi>0 &\sfa\xi\in\left(2k\pi-\frac{\pi}{4},\ 2k\pi+\frac{3\pi}{4}\right),\ k\in\bZ \\
\cos\xi+\sin\xi<0 &\sfa\xi\in\left(2k\pi+\frac{3\pi}{4},\ 2k\pi+\frac{7\pi}{4}\right),\ k\in\bZ \\
\cos\xi-\sin\xi>0 &\sfa\xi\in\left(2k\pi-\frac{3\pi}{4},\ 2k\pi+\frac{\pi}{4}\right),\ k\in\bZ \\
\cos\xi-\sin\xi<0 &\sfa\xi\in\left(2k\pi+\frac{\pi}{4},\ 2k\pi+\frac{5\pi}{4}\right),\ k\in\bZ
\end{align}
\end{subequations}
\begin{figure}[hht]
\center{\includegraphics[width=\textwidth]{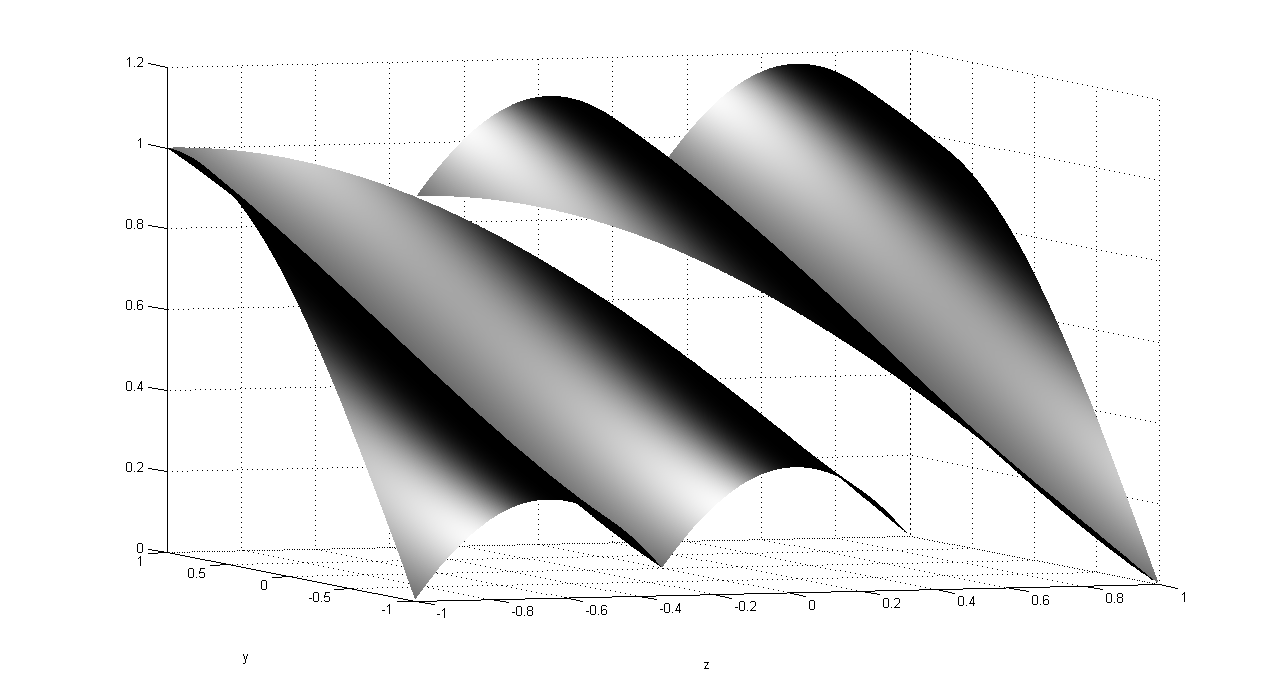}}\caption{Plot of the function $\ol G(z,y)$ for $\a=\frac{\pi}{4}$.}
\end{figure}\par
As we have seen, the Green's function $\ol G$ is not defined on the diagonal of $I^2$. For easier manipulation, we will define it in the diagonal as follows:
$$G(t,t)=\begin{cases}\lil_{s\to t^+}G(t,s)\quad\text{if}\quad m>0\\\lil_{s\to t^-}G(t,s)\quad\text{if}\quad m<0\end{cases}\text{for}\nkp t\in(-T,T);$$$$ \quad G(T,T)=\lil_{s\to T^-}G(s,s),\quad G(-T,-T)=\lil_{s\to -T^+}G(s,s)$$
Using expression (\ref{eqgb}) and formulae (\ref{trif}) we can prove the following theorem.
\begin{thm}\label{alphasign}\ \par
\begin{enumerate}
\item If $\a\in(0,\frac{\pi}{4})$ then $\ol G$ is strictly positive on $I^2$.
\item If $\a\in(-\frac{\pi}{4},0)$ then $\ol G$ is strictly negative on $I^2$.
\item If $\a=\frac{\pi}{4}$ then $\ol G$ vanishes on $P:=\{(-T,-T),(0,0),(T,T),(T,-T)\}$ and is strictly positive on $(I^2)\backslash P$.
\item If $\a=-\frac{\pi}{4}$ then $\ol G$ vanishes on $P$ and is strictly negative on $(I^2)\backslash P$.
\item If $\a\in\bR\backslash[-\frac{\pi}{4},\frac{\pi}{4}]$ then $\ol G$ is not positive nor negative on $I^2$.
\end{enumerate}
\end{thm}
\begin{proof}
Lemma \ref{Gop} allows us to restrict the proof to the positive values of $\a$.\par We study here the positive values of $\ol G(z,y)$ in $A:=\{(z,y)\in[-1,1]^2\ |\ z\ge|y|\}$. The rest of cases are done in an analogous fashion. Let
\begin{align*}
B_1 & :=\stackrel[k_1\in\bZ]{}{\bigcup}\(1-\frac{\pi}{\a}\(2k_1+\frac{3}{4}\),1-\frac{\pi}{\a}\(2k_1-\frac{1}{4}\)\),\ B_2:=\stackrel[k_2\in\bZ]{}{\bigcup}\frac{\pi}{\a}\(2k_2-\frac{1}{4},2k_2+\frac{3}{4}\),\\
C_1 & :=\stackrel[k_1\in\bZ]{}{\bigcup}\(1-\frac{\pi}{\a}\(2k_1+\frac{7}{4}\),1-\frac{\pi}{\a}\(2k_1+\frac{3}{4}\)\),\ 
C_2:=\stackrel[k_2\in\bZ]{}{\bigcup}\frac{\pi}{\a}\(2k_2+\frac{3}{4},2k_2+\frac{7}{4}\),\\
B & :=\{(z,y)\in B_1\times B_2\ |\ z>|y|\},\quad\text{and}\quad C:=\{(z,y)\in C_1\times C_2\ |\ z>|y|\}.
\end{align*}
Realize that $B\cap C=\emptyset$. Moreover, we have that $\ol G(z,y)> 0$ on $A$ if and only if $A\subset B\cup C$.\par To prove the case $A\subset B$, it is a necessary and sufficient condition that $[-1,1]\subset B_2$ and $[0,1]\subset{B_1}$.\par
$[-1,1]\subset B_2$ if and only if $k_2\in\frac{1}{2}(\frac{\a}{\pi}-\frac{3}{4},\frac{1}{4}-\frac{\a}{\pi})$ for some $k_2\in\bZ$, but, since $\a>0$, this only happens if $k_2=0$. In such a case $[-1,1]\subset\frac{\pi}{4\a}(-1,3)$, which implies $\a<\frac{\pi}{4}$. Hence, $\frac{\pi}{\a}>4$, so $[0,1]\subset(1-\frac{3}{4}\frac{\pi}{\a},1+\frac{1}{4}\frac{\pi}{\a})=\(1-\frac{\pi}{\a}\(2k_1+\frac{3}{4}\),1-\frac{\pi}{\a}\(2k_1-\frac{1}{4}\)\)$ for $k_1=0$. Therefore $A\subset B$.\par
We repeat this study for the case $A\subset C$ and all the other subdivisions of the domain of $\ol G$, proving the statement.
\end{proof}
The following definitions \cite{Cab3} lead to a direct corollary of Theorem \ref{alphasign}.
\begin{dfn} Let $\cF_\l(I)$ be the set of real differentiable functions $f$ defined on $I$ such that $f(-T)-f(T)=\l$. A linear operator $R:\cF_\l(I)\to L^1(I)$ is said to be
\begin{enumerate}
\item \textbf{strongly inverse positive on $\cF_\l(I)$} if $Rx\succ0\Ra x > 0\sfa x\in \cF_\l(I)$,
\item \textbf{strongly inverse negative on $\cF_\l(I)$} if $Rx\succ0\Ra x <0\sfa x\in \cF_\l(I)$.
\end{enumerate}
Where $x\succ0$ and $x\prec0$ stand for $x>0$ and $x<0$ a.e. respectively.
\end{dfn}

\begin{cor}\label{coralphasign} The operator $R_m:\cF_\l(I)\to L^1(I)$ defined as $R_m(x(t))=x'(t)+m\, x(-t)$, with $m\in\bR\backslash\{0\}$, satisfies
\begin{enumerate}
\item $R_m$ is strongly inverse positive if and only if $m\in(0,\frac{\pi}{4T}]$ and $\l\ge0$,
\item $R_m$ is strongly inverse negative if and only if $m\in[-\frac{\pi}{4T},0)$ and $\l\ge0$.

\end{enumerate}
\end{cor}
This last corollary establishes a maximum and anti-maximum principle (cf. \cite[Lemma~2.5, Remark~2.3]{Cab3}).
\par
The function $\ol G$ has a fairly convoluted expression which does not allow us to see in a straightforward way its dependence on $m$. This dependency can be analyzed, without computing and evaluating the derivative with respect to $m$, just using the properties of equation (\ref{eq1}a) in those regions where the operator $R$ is inverse positive or inverse negative. A different method to the one used here but pursuing a similar purpose can be found in \cite[Lemma 2.8]{Cab1} for the Green's function related to the second order Hill's equation.
\begin{pro} Let $G_{m_i}:I\to\bR$ be the Green's function and $u_i$ the solution to the problem (\ref{eq1}) with constant $m=m_i,\ i=1,2$ respectively. Then the following assertions hold.
\begin{enumerate}
\item If $0< m_1<m_2\le\frac{\pi}{4T}$ then $u_1>u_2$ on $I$ for every $h\succ 0$ and $G_{m_1}>G_{m_2}$ on $I^2$.
\item If $-\frac{\pi}{4T}\le m_1<m_2<0$ then $u_1>u_2$ on $I$ for every $h\succ 0$ and $G_{m_1}>G_{m_2}$ on $I^2$.
\end{enumerate}\par
\end{pro}
\begin{proof}
(1). Let $h\succ 0$ in equation (\ref{eq1}a). Then $u_i>0\ i=1,2$. We have that
$$u_i'(t)+m_iu_i(-t)=h(t)\quad i=1,2.$$

Therefore
\begin{equation*}\begin{aligned}
0=(u_2-u_1)'(t)+m_2u_2(-t)-m_1u_1(-t) & >(u_2-u_1)'(t)+m_1(u_2-u_1)(-t),\\
(u_2-u_1)(T) & -(u_2-u_1)(-T)=0,
\end{aligned}\end{equation*}
and hence, from Corollary \ref{coralphasign}, $u_2<u_1$ on $I$.\par
On the other hand, for all $t\in I$, it is satisfied that
\begin{equation}\label{u1u2} 
0>(u_{2}-u_{1})(t)=\int_{-T}^T(G_{m_2}(t,s)-G_{m_1}(t,s))h(s)\dif s\sfa h\succ0,
\end{equation}
This makes clear that $0\prec G_{m_2}\prec G_{m_1}$ a.e. on $I^2$.\par
To prove that $G_{m_2}<G_{m_1}$ on $I^2$, let $s\in I$ be fixed, and define $v_i: \bR \to \bR$ as the $2\, T$-- periodic extension to the whole real line of $G_{m_i}(\cdot,s)$.\par
Using $(I')$ -- $(IV')$, we have that $v_2-v_1$ is a continuosly differentiable function on $I_s \equiv (s,s+2\,T)$. Futhermore, it is clear that $(v_2-v_1)'$ is absolutely continuous on $I_s$. Using $(III')$, we have that $$(v_2-v_1)'(t)+m_2v_2(-t)-m_1v_1(-t) =0\quad\text{on } I_s.$$ As consequence,
 $v_i''(t)+m_i^2\,v_i(t)=0$ a.e. on $I_s$. Moreover, using $(II')$ and $(IV')$ we know that 
$$(v_2-v_1)(s)=(v_2-v_1)(s+2\,T), \quad (v_2-v_1)'(s)=(v_2-v_1)'(s+2\,T).$$

Hence, for all $t \in I_s$, we have that
$$0=(v_2-v_1)''(t)+m_2^2\,v_2(t)-m_1^2\,v_1(t) >(v_2-v_1)''(t)+m_1^2\,(v_2-v_1)(t).$$

The periodic boundary value conditions, together the fact that for this range of values of $m_1$, operator $v''+m^2_1\, v$ is strongly inverse positive (see \cite{Cab3, Cab1}), we conclude that $v_2<v_1$ on $I_s$, this is, $G_{m_2}(t,s)<G_{m_1}(t,s)$ for all $t,\, s \in I$.\par
(2). This is straightforward using part (1), Lemma \ref{Gop} and Theorem \ref{alphasign}:
$$G_{m_2}(t,s)=-G_{-m_2}(-t,-s)<-G_{-m_1}(-t,-s)=G_{m_1}(t,s)<0\sfa t,s\in I.$$
By equation (\ref{u1u2}), $u_2<u_1$ on $I$.\par
\end{proof}
\begin{rem} In (1) and (2) we could have added that $u_1<u_2\sfa h\prec0$. These are straightforward consequences of the rest of the proposition.
\end{rem}
\section{Applications}

This section is devoted to point out some applications of the given results to the existence of solutions of nonlinear periodic boundary value problems. Due to the fact that the proofs follow similar steps to the ones given in previous papers, we omit them.

\subsection{Lower and upper solutions method}
Lower and upper solutions methods are a variety of widespread techniques that supply information about the existence --and sometimes construction-- of solutions of differential equations. Depending on the particular type of differential equation and the involved boundary value conditions, it is subject to these techniques change but are in general suitable --with proper modifications-- to other cases.\par
For this application we will follow the steps in \cite{Cab3} and use Corollary \ref{coralphasign} to establish conditions under which the more general problem
\begin{equation}\label{eqgenpro}
x'(t)  =f(t,x(-t))\sfa t\in I,\quad
x(-T) =x(T),
\end{equation}
has a solution. Here $f:I\times\bR\to\bR$ is a \textbf{Carath\'eodory function}, that is, $f(\cdot,x)$ is measurable for all $x\in\bR$, $f(t,\cdot)$ is continuous for a.e. $t\in I$, and for every $R >0$, there exists $h_R\in L^1(I)$ such that, if with $|x|<R$ then
$$|f(t,x)|\le h_R(t)\quad \text{for a.e. }t\in I.$$

\begin{dfn}We say that $\a\in AC(I)$ is a \textbf{lower solution} of (\ref{eqgenpro}) if $\a$ satisfies
\begin{equation*}
\a'(t)  \ge f(t,\a(-t))\quad \text{for a.e. } t\in I,\quad
\a(-T)-\a(T)\ge 0.
\end{equation*}
\end{dfn}
\begin{dfn}We say that $\b\in AC(I)$ is an \textbf{upper solution} of (\ref{eqgenpro}) if $\b$ satisfies
\begin{equation*}
\b'(t) \le f(t,\b(-t))\quad \text{for a.e. } t\in I,\quad
\b(-T)-\b(T)\le 0.
\end{equation*}
\end{dfn}

We establish now a theorem that proves the existence of solutions of (\ref{eqgenpro}) under some conditions. The proof follows the same steps of \cite[Theorem 3.1]{Cab3} and we omit it here.

\begin{thm} \label{thmlu}Let $f:I\times\bR\to\bR$ be a Carath\'eodory function. If there exist $\a\ge\b$ lower and upper solutions of (\ref{eqgenpro}) respectively and $m\in(0,\frac{\pi}{4T}]$ such that
\begin{equation}\label{condul1}
f(t,x)-f(t,y)\ge-m(x-y)\quad\text{for a.e. } t\in I\text{ with }\b(t)\le y\le x\le\a(t),
\end{equation}
then there exist two monotone sequences $(\a_n)_\n$, $(\b_n)_\n$, non-increasing and non-decreasing respectively, with $\a_0=\a$, $\b_0=\b$, which converge uniformly to the extremal solutions in $[\b,\a]$ of (\ref{eqgenpro}).\par
Furthermore, the estimate $m=\frac{\pi}{4T}$ is best possible in the sense that there are problems with its unique solution outside of the interval $[\b,\a]$ for $m>\frac{\pi}{4T}$.
\end{thm}\par
In an analogous way we can prove the following theorem.

\begin{thm} \label{thmlu2}Let $f:I\times\bR\to\bR$ be a Carath\'eodory function. If there exist $\a\le\b$ lower and upper solutions of (\ref{eqgenpro}) respectively and $m\in[-\frac{\pi}{4T},0)$ such that
\begin{equation}\label{condul2}f(t,x)-f(t,y)\le-m(x-y)\quad\text{for a.e. } t\in I\text{ with }\a(t)\le y\le x\le\b(t),\end{equation}
then there exist two monotone sequences $(\a_n)_\n$, $(\b_n)_\n$, non-decreasing and non-increasing respectively, with $\a_0=\a$, $\b_0=\b$, which converge uniformly to the extremal solutions in $[\a,\b]$ of (\ref{eqgenpro}).\par
Furthermore, the estimate $m=-\frac{\pi}{4T}$ is best possible in the sense that there are problems with its unique solution outside of the interval $[\a,\b]$ for $m<-\frac{\pi}{4T}$.
\end{thm}

\subsection{Existence of solutions via Krasnosel'ski\u{\i} fixed point theorem}

In this section we implement the methods used in \cite{Kra} for the existence of solutions of second order differential equations to prove new existence results for problem 
\begin{equation}\label{eqgenpro2}
x'(t)  =f(t,x(-t),x(t))\sfa t\in I,\quad
x(-T) =x(T),
\end{equation}
where $f:I\times\bR\times\bR\to\bR$ is an $L^1$ Carath\'eodory function and $2T$-periodic on $t$.\par
Let us first establish the fixed point theorem we are going to use \cite{Kra}.
\begin{thm}[Krasnosel'skii]\label{fpt}
Let $\cB$ be a Banach space, and let $\cP\subset\cB$ be a cone in $\cB$. Assume $\Omega_1$, $\Omega_2$ are open subsets of $\cB$ with $0\in\Omega_1$, $\ol \Omega_1\subset\Omega_2$ and let $A:\cP\cap(\ol{\Omega_2}\backslash\Omega_1)\to\cP$ be a completely continuous operator such that one of the following conditions is satisfied:
\begin{enumerate}
\item $\|Au\|\le\|u\|$ if $u\in\cP\cap\partial\Omega_1$ and $\|Au\|\ge\|u\|$ if $u\in\cP\cap\partial\Omega_2$,
\item $\|Au\|\ge\|u\|$ if $u\in\cP\cap\partial\Omega_1$ and $\|Au\|\le\|u\|$ if $u\in\cP\cap\partial\Omega_2$.
\end{enumerate}
Then, $A$ has at least one fixed point in $\cP\cap(\ol{\Omega_2}\backslash\Omega_1)$.
\end{thm}
In the following, let $m\in\bR\backslash\{0\}$ and $G$ be the Green function for problem
$$x'(t)+m\, x(-t)=h(t),\quad x(-T)=x(T).$$
Let $M=\sup\{G(t,s)\ :\ t,s\in I\}$, \quad $L=\inf\{G(t,s)\ :\ t,s\in I\}$.

\begin{thm}
Let $m\in(0,\frac{\pi}{4T})$. Assume there exist $r$, $R\in\bR^+$, $r<R$ such that
\begin{equation}\label{eqci}
f(t,x,y)+m\, x\ge0\sfa x,y\in\left[\frac{L}{M}r,\frac{M}{L}R\right]\text{, a.e. } t\in I.
\end{equation}
Then, if one of the following conditions holds,
\begin{enumerate}
\item
\begin{align*}
f(t,x,y)+m\, x & \ge\frac{M}{2TL^2}x\sfa x,y\in\left[\frac{L}{M}r,r\right]\text{, a.e. } t\in I,\\
f(t,x,y)+m\, x & \le\frac{1}{2TM}x\sfa x,y\in\left[R,\frac{M}{L}R\right]\text{, a.e. } t\in I;
\end{align*}
\item
\begin{align*}
f(t,x,y)+m\, x & \le\frac{1}{2TM}x\sfa x,y\in\left[\frac{L}{M}r,r\right]\text{, a.e. } t\in I,\\
f(t,x,y)+m\, x & \ge\frac{M}{2TL^2}x\sfa x,y\in\left[R,\frac{M}{L}R\right]\text{, a.e. } t\in I;
\end{align*}
problem (\ref{eqgenpro2}) has a positive solution.
\end{enumerate}
\end{thm}
\par
If $\cB=(\cC(I),\|\cdot\|_\infty)$, by defining the absolutely continuous operator $A:\cB\to\cB$ such that $$(A\,x)(t):=\int_{-T}^{T}G(t,s)[f(s,x(-s),x(s))+m\, x(-s)]\dif s$$
we deduce the result following the same steps as in \cite{Tor}.

\par
We present now two corollaries (analogous to the ones in \cite{Tor}). The first one is obtained by strengthening the hypothesis and making them easier to check.
\begin{cor}Let $m\in(0,\frac{\pi}{4T})$, $f(t,x,y)\ge0$ for all $x,y\in\bR^+$ and a.e. $t\in I$. Then, if one of the following condition holds:
\begin{enumerate}\label{coreasy}
\item $$\lim_{x,y\to 0^+}\frac{f(t,x,y)}{x}=+\infty,\quad \lim_{x,y\to +\infty}\frac{f(t,x,y)}{x}=0,$$
\item $$\lim_{x,y\to 0^+}\frac{f(t,x,y)}{x}=0,\quad \lim_{x,y\to +\infty}\frac{f(t,x,y)}{x}=+\infty$$
\end{enumerate}
uniformly for a.e. $t\in I$, then problem (\ref{eqgenpro2}) has a positive solution.
\end{cor}
The second corollary is obtained just making the change of variables $y=-x$.
\begin{cor}\label{cor1}
Let $m\in(0,\frac{\pi}{4T})$. Assume there exist $r$, $R\in\bR^+$, $r<R$ such that
\begin{equation}
f(t,x,y)+m\, x\le0\sfa x,y\in\left[-\frac{M}{L}R,-\frac{L}{M}r\right]\text{, a.e. } t\in I.
\end{equation}
Then, if one of the following conditions holds,
\begin{enumerate}
\item
\begin{align*}
f(t,x,y)+m\, x & \le\frac{M}{2TL^2}x\sfa x,y\in\left[-r,-\frac{L}{M}r\right]\text{, a.e. } t\in I,\\
f(t,x,y)+m\, x & \ge\frac{1}{2TM}x\sfa x,y\in\left[-\frac{M}{L}R,-R\right]\text{, a.e. } t\in I;
\end{align*}
\item
\begin{align*}
f(t,x,y)+m\, x & \ge\frac{1}{2TM}x\sfa x,y\in\left[-r,-\frac{L}{M}r\right]\text{, a.e. } t\in I,\\
f(t,x,y)+m\, x & \le\frac{M}{2TL^2}x\sfa x,y\in\left[-\frac{M}{L}R,-R\right]\text{, a.e. } t\in I;
\end{align*}
problem (\ref{eqgenpro2}) has a negative solution.
\end{enumerate}
\end{cor}
\par
Similar results to these --with analogous proofs-- can be given when the Green's function is negative.

\begin{thm}\label{teo2}
Let $m\in(-\frac{\pi}{4T},0)$. Assume there exist $r$, $R\in\bR^+$, $r<R$ such that
\begin{equation}
f(t,x,y)+m\, x\le0\sfa x,y\in\left[\frac{M}{L}r,\frac{L}{M}R\right]\text{, a.e. } t\in I.
\end{equation}
Then, if one of the following conditions holds,
\begin{enumerate}
\item
\begin{align*}
f(t,x,y)+m\, x & \le\frac{L}{2TM^2}x\sfa x,y\in\left[\frac{M}{L}r,r\right]\text{, a.e. } t\in I,\\
f(t,x,y)+m\, x & \ge\frac{1}{2TL}x\sfa x,y\in\left[R,\frac{L}{M}R\right]\text{, a.e. } t\in I;
\end{align*}
\item
\begin{align*}
f(t,x,y)+m\, x & \ge\frac{1}{2TL}x\sfa x,y\in\left[\frac{M}{L}r,r\right]\text{, a.e. } t\in I,\\
f(t,x,y)+m\, x & \le\frac{L}{2TM^2}x\sfa x,y\in\left[R,\frac{L}{M}R\right]\text{, a.e. } t\in I;
\end{align*}
problem (\ref{eqgenpro2}) has a positive solution.
\end{enumerate}
\end{thm}
\begin{cor}\label{cor2}
Let $m\in(-\frac{\pi}{4T},0)$. Assume there exist $r$, $R\in\bR^+$, $r<R$ such that
\begin{equation}
f(t,x,y)+m\, x\ge0\sfa x,y\in\left[-\frac{L}{M}R,-\frac{M}{L}r\right]\text{, a.e. } t\in I.
\end{equation}
Then, if one of the following conditions holds,
\begin{enumerate}
\item
\begin{align*}
f(t,x,y)+m\, x & \ge\frac{L}{2TM^2}x\sfa x,y\in\left[-r,-\frac{M}{L}r\right]\text{, a.e. } t\in I,\\
f(t,x,y)+m\, x & \le\frac{1}{2TL}x\sfa x,y\in\left[-\frac{L}{M}R,-R\right]\text{, a.e. } t\in I;
\end{align*}
\item
\begin{align*}
f(t,x,y)+m\, x & \le\frac{1}{2TL}x\sfa x,y\in\left[-r,-\frac{M}{L}r\right]\text{, a.e. } t\in I,\\
f(t,x,y)+m\, x & \ge\frac{L}{2TM^2}x\sfa x,y\in\left[-\frac{L}{M}R,-R\right]\text{, a.e. } t\in I;
\end{align*}
problem (\ref{eqgenpro2}) has a negative solution.
\end{enumerate}
\end{cor}\par
We could also state analogous corollaries to Corollary \ref{coreasy} for Theorem \ref{teo2} and Corollaries \ref{cor1} and \ref{cor2}.
\subsection{Examples}
We will now analyze two examples to which we can apply the previous results. Observe that both examples do not lie under the hypothesis of the existence results for bounded solutions for differential equations with reflection of the argument in \cite{Wie2} nor in those of the more general results found in \cite{Wie, Wat1, Wat2, Sha, Aft} or any other existence results known to the authors.\par

\begin{exa}\label{exa3} Consider the problem
\begin{equation}\label{exa3.1}
x'(t)  =\lambda \, \sinh{(t-x(-t))}, \, \sfa\, t\in I,\quad
x(-T) =x(T).
\end{equation}
 It is easy to check that $\a\equiv T$ and $\b\equiv -T$ are lower and upper solutions for problem (\ref{exa3.1}) for all $\lambda \ge 0$.  Since $f(t,y):=\lambda\, \sinh{(t-y)}$ satisfies that $|\frac{\partial f}{\partial y}(t,y)|\le \lambda \cosh{(2\, T)}$, for all $(t,y) \in I^2$, we know, from Theorem \ref{thmlu}, that problem (\ref{exa3.1}) has extremal solutions on $[-T,T]$ for all
 $$\displaystyle 0 \le \lambda \le \frac{ \pi}{4\,T\, \cosh{(2\, T)}}.$$
\end{exa}

\begin{exa}Consider the problem
\begin{equation}\label{exa2}
x'(t)  =t^2\,x^2(t)[\cos^2(x^2(-t))+1]\sfa t\in I,\quad
x(-T) =x(T).
\end{equation}
\par
By defining $f(t,x,y)$ as the $2T$-periodic extension on $t$ of the function $t^2x^2[\cos^2(y^2)+1]$, we may to apply Corollary \ref{coreasy} to deduce that problem (\ref{exa2}) has a positive solution. Using the analogous corollary for Corollary \ref{cor2}, we know that it also has a negative solution.
\end{exa}

\section*{Acknowledgements}
We want to thank to the anonymous referee for his/her useful advice and careful revision of this article and for the suggestions given to improve the paper.


\end{document}